\documentclass[a4paper, 13pt, oneside]{article}
\usepackage{geometry} 
\usepackage{graphicx}
\graphicspath{immagini/}
\usepackage{multicol} 
\usepackage{ragged2e} 
\usepackage{amsthm}
\usepackage{amssymb}
\usepackage{faktor}
\usepackage[english]{babel} 

\usepackage{minitoc} 
\nomtcrule 
\addto{\captionsitalian}{
  
}

\usepackage{comment}

\usepackage{xpatch}
\usepackage{blindtext}

\makeatletter

\xpatchcmd{\@makeschapterhead}{%
  \Huge \bfseries  #1\par\nobreak%
}{%
  \Huge \bfseries\centering #1\par\nobreak%
}{\typeout{Patched makeschapterhead}}{\typeout{patching of @makeschapterhead failed}}

\xpatchcmd{\@makechapterhead}{%
  \huge\bfseries \@chapapp\space \thechapter
}{%
  \huge\bfseries\centering \@chapapp\space \thechapter
}{\typeout{Patched @makechapterhead}}{\typeout{Patching of @makechapterhead failed}}

\usepackage{fancyhdr}
\usepackage[export]{adjustbox}

\usepackage{hyperref} 
\hypersetup{
    colorlinks=true,
    citecolor=cyan,
    linkcolor=black,
    urlcolor=black,
    pdftitle=Una buona tesi,
    pdfauthor=Un povero studente di medicina,
    }

\usepackage{booktabs} 
\usepackage{multirow}
\usepackage[table,xcdraw]{xcolor}
\usepackage{graphicx}

\usepackage{lscape}
\usepackage{float}
\usepackage{wrapfig}

\makeatletter
\newenvironment{Proof}[1][\proofname]{\proof[#1]\mbox{}\\*}{\endproof}
\makeatother

\newtheoremstyle{break}
  {}
  {}
  {\itshape}
  {}
  {\bfseries}
  {}
  {\newline}
  {}

\theoremstyle{break}
\newtheorem{teor}{Theorem}[section]
\newtheorem{cor}{Corollary}[section]
\newtheorem{lem}{Lemma}[section]

\usepackage{amsmath,amssymb,mathrsfs}

\title{\textbf{Groups with Finiteness Conditions on commutator subgroups}}
\author{
    Rosa Cascella \\
    \textit{Dipartimento di Matematica e Applicazioni} \\
\textit{Università di Napoli Federico II} \\
    \textit{Napoli, Italy} \\
    \texttt{rosa.cascella@unina.it} 
}
\date{}

\begin{document}
\maketitle

\begin{abstract}
    \noindent The structure of groups for which certain sets of commutator subgroups are finite is investigated, with a particular focus on the relationship between these groups and those with finite derived subgroup.
\end{abstract}


\section{Introduction}

A finiteness condition is a group theoretical property satisfied by every finite group. In many cases, the imposition of a finiteness condition produces strong restrictions on the structure of a group. For instance, this is the case of the property $FC$, which is defined requiring that each element of the group has only finitely many conjugates. The behavior of $FC$-groups has been investigated by several authors over the last seventy years; among the basic properties of $FC$-groups, it is well known that the commutator subgroup of any $FC$-group is locally finite and that finitely generated groups with the property $FC$ are central-by-finite. A special subclass of $FC$-groups is formed by finite-by-abelian groups. A relevant Theorem of Neumann \cite{neu} states that a group $G$ has finite derived subgroup if and only if it has boundedly finite conjugacy classes; moreover, Neumann \cite{neu2} proved that a group $G$ is finite-by-abelian if and only if each subgroup of $G$ has finite index in its normal closure. Our aim in this paper is to study new finiteness conditions related to the derived subgroup of a group. In \cite{ci} are described groups with finitely many derived groups of subgroups or infinite subgroups. In this paper the properties $\Bar{\mathcal{K}}$ and $\Bar{\mathcal{K}}_{\infty}$ are introduced. We shall say that a group $G$ has the property $\Bar{\mathcal{K}}$ if the set $\{[X,H]\ |\ H\leq G\}$ is finite for each subgroup $X$ of $G$. Groups with $\Bar{\mathcal{K}}$ need not have finite derived subgroup, since any Tarski group (i.e. infinite group whose proper non-trivial subgroups have prime orders) has the property $\Bar{\mathcal{K}}$. On the other hand, it turns out that within the class of locally (soluble-by-finite) groups, the property $\Bar{\mathcal{K}}$ is equivalent to the finiteness of the derived subgroup. 

\medskip

\noindent \textbf{Theorem A.} A group $G$ is finite-by-abelian if and only if it is locally (soluble-by-finite) and has the property $\bar{\mathcal{K}}$.

\medskip

A group $G$ has the more general property $\bar{\mathcal{K}}_{\infty}$ if the set $\{[X,H] \,|\,H\leq G, H \ is \ infinite\}$ is finite for each subgroup $X$ of $G$. Notice that the property $\bar{\mathcal{K}}_{\infty}$ does not imply the finiteness of the derived subgroup, even in the case of soluble groups, as is shown by the locally dihedral 2-group; in particular, the class of $\bar{\mathcal{K}}_{\infty}$-groups is wider than the class of $\bar{\mathcal{K}}$-groups. However, our second main result shows that $\bar{\mathcal{K}}_{\infty}$-groups with infinite derived subgroup have a very restricted structure. 

\medskip

\noindent \textbf{Theorem B.} A soluble-by-finite group $G$ has the property $\bar{\mathcal{K}}_{\infty}$ if and only if it has finite commutator subgroup or it is a finite extension of a group of the type $p^{\infty}$ for some prime number $p$.

\medskip

Finally, we mention that groups with finitely many derived groups of non-normal subgroups or of infinite non-normal subgroups are studied in \cite{cern}. In \cite{kappa} it is studied the behavior of groups in which the set $\{[x,H] \ | \ H\leq G\}$ (or $\{[x,H] \ | \ H\leq G, \, H \ is \ infinite \}$) is finite for each element $x$ of $G$. Most of the notation is standard and can be found in \cite{robinson}.



\section{Basic properties of $\Bar{\mathcal{K}}$ and $\bar{\mathcal{K}}_{\infty}$-groups}

Clearly, the properties $\Bar{\mathcal{K}}$ and $\bar{\mathcal{K}}_{\infty}$ are inherited by subgroups and homomorphic images. Moreover, if $N$ is any infinite normal subgroup of a $\bar{\mathcal{K}}_{\infty}$-group, then $G/N$ is a $\bar{\mathcal{K}}$-group.

\begin{lem}\label{abel}
    Let $G$ be a $\bar{\mathcal{K}}_{\infty}$-group, and let $A$ be a torsion-free abelian normal subgroup of $G$. Then $A$ is contained in the center $Z(G)$ of $G$.
\end{lem}

\begin{Proof}
    Let $a$ be an element of $A$ and $x$ an element of $G$. Then 
    \begin{center}
        $[<x>,<a^n>]=[x,a^n]^{<x>}=([x,a]^n)^{<x>}=([x,a]^{<x>})^n.$
               
    \end{center}
    It follows that the set $\{([x,a]^{<x>})^{r!} \, | \, r\in \mathbb{N}_0\}$ is finite. Consequently, there exists $k\in \mathbb{N}_0$ such that $$([x,a]^{<x>})^{k!}=([x,a]^{<x>})^{(k+1)!}=....$$ Thus the subgroup $([x,a]^{<x>})^{k!}$ is divisible. On the other hand, $<x,a>$ is a finitely generated metabelian group, so that it is residually finite and hence $[x,a]^{k!}=1$. Therefore, $[x,a]=1$ and $a\in Z(G)$.
\end{Proof}

\begin{lem}\label{fg}
    If $G$ is a $\Bar{\mathcal{K}}$-group, then for every subgroup $H$ of $G$ there exists a subgroup $N$ of $H$ such that $N$ is finitely generated and $N'=H'$. In particular, $G'$ is contained in a finitely generated subgroup of $G$. 
\end{lem}

\begin{proof}
    Let $H$ be a subgroup of $G$, then $$H'=<[<h>,H]\, |\,h\in H>=<[<h_1>,H],...,[<h_t>,H]>$$ for some $h_1,...,h_t\in H$. Moreover, there exist $k^1_1,...,k^1_{s_1},...,k^t_1,...,k^t_{s_t}\in H$ such that, fixed $i\in \{1,...,t\}$ it holds that
    
    \begin{center}
         $[<h_i>,H]=<[<h_i>,<k>]\,|\,k\in H>$
         $=<[<h_i>,<k^i_1>],...,[<h_i>,<k_{s_i}^i>]>$
         $\leq <<h_i,k^i_1>',...,<h_i,k_{s_i}^i>'>$
         $\leq <h_i,k^i_1,...,k_{s_i}^i>'.$
    \end{center}
    Thus $H'\leq <h_1,...,h_t,k^1_1,...,k_{s_t}^t>'.$
\end{proof}

We note here that the proof of Lemma \ref{fg} still works if the hypothesis that $G$ is a $\bar{\mathcal{K}}$-group is replaced by the requirement that it is a torsion-free $\bar{\mathcal{K}}_{\infty}$-group. As a consequence of the previous Lemma, we have that given a $\Bar{\mathcal{K}}$-group $G$ such that $|G:G'|$ is finite, then $G'$ is finitely generated. We will show now that from the previous results it follows that if $G$ is a locally (soluble-by-finite) $\Bar{\mathcal{K}}$-group, then $G'$ is periodic.

\begin{lem}\label{per}
    Let $G$ be a locally (soluble-by-finite) $\Bar{\mathcal{K}}$-group. Then the commutator subgroup $G'$ of $G$ is periodic.
\end{lem}

\begin{proof}
    Since $G$ verifies the property $\Bar{\mathcal{K}}$, there exists a finitely generated group $N$ such that $N'=G'$. Thus, without loss of generality, we may suppose that $G$ is finitely generated. Assume for a contradiction that the statement is false, and let $T$ be the largest periodic normal subgroup of $G$. Clearly the factor group $G/T$ is also a counterexample, so we may suppose without loss of generality that $G$ has no periodic non-trivial normal subgroups. Among all such counterexamples choose $G$ in such a way that its soluble radical $S$ has minimal derived length $m$. Let $U$ be the smallest non-trivial term of the derived series of $S$, and let $A$ be a maximal abelian normal subgroup of $G$ containing $U$. Clearly, $A$ is torsion-free and so it is contained in $Z(G)$ by Lemma \ref{abel}. If $L/A$ is the largest periodic normal subgroup of $G/A$, the factor group $L/Z(L)$ is locally finite, so that $L'$ is locally finite by Schur's theorem and hence $L'=\{1\}$. Thus, $L$ is abelian, so that $L=A$, and hence $G/A$ has no periodic non-trivial normal subgroups. On the other hand, the soluble radical of $G/A$ has derived length less then $m$, and so the minimal choice of $G$ yields that $G/A$ is abelian. Therefore, $G$ is nilpotent, and hence $A=C_G(A)=G$. This contradiction completes the proof of the Lemma.
\end{proof}

\section{$\bar{\mathcal{K}}$-groups}

The first result of this section shows that locally (soluble-by-finite) $\Bar{\mathcal{K}}$-groups have finite second commutator subgroup. 

\begin{lem}\label{period}
    Let $G$ be a periodic group which satisfies $\Bar{\mathcal{K}}$. Then $G'$ is finitely generated.
\end{lem}
\begin{Proof}
    The derived subgroup of the group $G$ is the subgroup $<[<g>,G]\, |\ g\in G>$ and the set $\{[<g>,G]\, |\,\ g\in G\}$ is finite. Therefore they exist $g_1,...,g_n\in G$ such that $$\{[<g>,G]\, |\,\ g\in G\}=\{[<g_1>,G],...,[<g_t>,G]\}.$$ The subgroup $[<g>,G]$ is equal to $<[<g>,<x>] \, | \, x\in G>$, so, fixed $i$ in the set $\{1,...,t\}$, there exist $x_1^i,...,x_{s_i}^i$ such that $$<[<g>,<x>] \, | \, x\in G>=\{[<g_i>,<x_1^i>],...,[<g_i>,<x_{s_i}^i>]\},$$ where $[<g_i>,<x>]=<[g_i^n,x^m]\, | \, n,m\in \mathbb{Z}>$. $G$ is periodic, and so the subgroup $[<g_i>,<x>]$ is finitely generated for every $x\in G$. Consequentially, $G'$ is finitely generated.
\end{Proof}

\begin{lem}\label{finito}

    Let G be a group. If G is a locally (soluble-by-finite) $\Bar{\mathcal{K}}$-group, then $G''$ is finite. 
\end{lem}
\begin{proof}
    By Lemmas \ref{per} and \ref{period} $G''$ is finitely generated and periodic, therefore it is finite. 
\end{proof}

\begin{lem}\label{rf}
    Let $G$ be a residually finite $\bar{\mathcal{K}}_{\infty}$-group. Then $G'$ is finite.
\end{lem}
\begin{proof}
    Assume that $G$ is infinite. Let $\mathcal{L}$ be the set of all normal subgroups of finite index of $G$. There exists a finite subset $\mathcal{F}$ of $\mathcal{L}$ such that $$\{[G,H] \,|\, H\in \mathcal{L}\}=\{[G,H] \,|\, H\in \mathcal{F}\}.$$ Then $$[G,\bigcap_{H\in \mathcal{F}}H]\leq \bigcap_{H\in \mathcal{F}} [G,H]=\bigcap_{H\in \mathcal{L}} [G,H]\leq \bigcap_{H\in \mathcal{L}} H=\{1\}$$ and hence $\bigcap_{H\in \mathcal{F}}H$ is contained in the center $Z(G)$. As the index $$|G:\bigcap_{H\in \mathcal{F}}H|$$ is finite, it follows that $G'$ is finite.
\end{proof}
\medskip

We are ready to prove the main result of this section.

\medskip

\noindent \textbf{Theorem A}

\noindent A group $G$ is finite-by-abelian if and only if it is locally (soluble-by-finite) and has the property $\bar{\mathcal{K}}$.

\begin{proof}
    
The necessary condition is obviously verified. Vice versa, assume that $G$ is a locally (soluble-by-finite) $\bar{\mathcal{K}}$-group and observe that $G''$ is finite by Lemma \ref{finito}, so that, replacing $G$ by $G/G''$ it is possible to suppose that $G$ is metabelian. As $G$ satisfies $\Bar{\mathcal{K}}$, there exists a finitely generated subgroup $N$ of $G$ such that $N'=G'$. So, without loss of generality, it is possible to assume that $G$ is finitely generated. It follows that $G$ is residually finite, and so $G'$ is finite by Lemma \ref{rf}.

\end{proof}

We conclude this section with a result on perfect $\Bar{\mathcal{K}}$-groups.

\begin{teor}
    Let $G$ be a perfect group with the property $\Bar{\mathcal{K}}$. Then $G$ is finitely generated.
\end{teor}
\begin{proof}
    We see from Lemma \ref{fg} that there exists a finitely generated subgroup $N$ such that $G=G'=N'$ and hence $G=N$ is finitely generated.  
\end{proof}

\section{$\bar{\mathcal{K}}_{\infty}$-groups}

\begin{lem}\label{centro}
    Let $G$ be a $\bar{\mathcal{K}}_{\infty}$-group. If $Z(G)$ is infinite, then $G$ satisfies $\bar{\mathcal{K}}$.
\end{lem}

\begin{proof}
    Let $X$ be a subgroup of $G$, then $[X,H]=[X,HZ(G)]$ for every subgroup $H$ of $G$, and hence the set $$\{[X,H]\,|\, H\leq G\}=\{[X,HZ(G)]\,|\, H\leq G\}$$ is finite. Therefore, $G$ is a $\bar{\mathcal{K}}$-group.
\end{proof}

\begin{cor}\label{cenk}
    Let $G$ be a locally (soluble-by-finite) $\bar{\mathcal{K}}_{\infty}$-group. If $Z(G)$ is infinite, then $G'$ is finite.
\end{cor}

\begin{proof}
    The result is an immediate consequence of Lemma \ref{centro} and Theorem A.
\end{proof}

As a consequence of Lemmas \ref{abel} and \ref{centro}, it turns out that for groups containing an infinite torsion-free abelian normal subgroup the properties $\bar{\mathcal{K}}$ and $\bar{\mathcal{K}}_{\infty}$ are equivalent.

\begin{lem}
    Let $G$ be a soluble-by-finite $\bar{\mathcal{K}}_{\infty}$-group and let $N$ be an infinite normal subgroup of $G$. Then $G/N$ has finite commutator subgroup.
\end{lem}

\begin{proof}
    The factor group $G/N$ has the property $\bar{\mathcal{K}}$, and hence it has finite commutator subgroup by Theorem A.
\end{proof}

\begin{cor}\label{nilp}
    Let $G$ be an infinite nilpotent group with the property $\bar{\mathcal{K}}_{\infty}$. Then $G'$ is finite.
\end{cor}

\begin{proof}
    Let $i$ be the smallest positive integer such that the $i$-th term $Z_i(G)$ of the upper central series of $G$ is infinite. Then $G/Z_{i-1}(G)$ is a $\bar{\mathcal{K}}_{\infty}$-group with infinite center, and hence it has the property $\bar{\mathcal{K}}$ by Lemma \ref{centro}. It follows from Theorem A that $\frac{G'Z_{i-1}(G)}{Z_{i-1}(G)}$ is finite, so $G'$ is finite since $Z_{i-1}(G)$ is finite.
\end{proof}

\begin{cor}\label{meta}

    Let $G$ be a metabelian non-periodic $\bar{\mathcal{K}}_{\infty}$-group. Then $G'$ is finite.
\end{cor}
\begin{proof}
    Let $a$ be an element of infinite order of $G$. If $x$ and $y$ are element of $G$, the subgroup $<x,y,a>$ is residually finite, and hence $<x,y,a>'$ is finite. On the other hand, there exist finitely many elements $y_1,...,y_k,x_1^1,...,x^1_{s_1},...,x^k_1,...,x^k_{s_k}$ such that $$\{[<x,a>,<y,a>]\,|\, y\in G\}=\{[<x,a>,<y_1,a>],...,[<x,a>,<y_k,a>]\}$$ and for each $i\in \{1,...,k\}$ $$\{[<x,a>,<y_i,a>]\,|\, x\in G\}=\{[<x_1^i,a>,<y_i,a>],...,[<x^i_{s_i},a>,<y_i,a>]\}.$$ Therefore $G'=<[<x,a>,<y,a>] \,|\, x,y\in G>$ is finite.

\end{proof}

\begin{lem}\label{4.5}
    Let $G$ be a $\bar{\mathcal{K}}_{\infty}$-group and let $A$ be an abelian normal subgroup of $G$ of prime exponent $p$. If $g$ is any element of finite order of $G$, the subgroup $[g,A]$ is finite.
\end{lem}

\begin{proof}
    Assume for a contradiction that the statement is false, so there exists an element $g$ of $G$ with prime-power order $q^n$ such that the subgroup $[g,A]$ is infinite. If $q=p$, the subgroup $<g,A>$ is nilpotent (see [\cite{robinson}, Part 2, Lemma 6.34]), so that, by Corollary \ref{nilp}, $<A,g>'$ is finite, and hence $[g,A]$ is finite. This contradiction shows that $q\neq p$. Then $A$ is a completely reducible $<g>$-module, and so $$A=\underset{i \in I}{\mathrm{Dr}} \, A_i$$ where each $A_i$ is a finite $<g>$-invariant subgroup of $A$ on which $g$ acts irreducibly. Let $I'$ be the subset of $I$ consisting of all elements $i$ such that $[<g>,A_i]=[g,A_i]\neq \{1\}$. Then $I'$ is infinite and $[<g>,A_i]=[g,A_i]=A_i$ for each $i\in I'$. It follows that $$[<g>,<A_j\,|\, j\in J>]=\underset{j \in J}{\mathrm{Dr}} \, A_j$$ for every subset $J$ of $I'$. On the other hand, $I'$ contains infinitely many infinite subsets, contradicting the condition $\bar{\mathcal{K}}_{\infty}$.
\end{proof}

\begin{lem}\label{4.7}
    Let $G$ be a soluble group with no non-trivial periodic divisible abelian normal subgroups. If $G$ has the property $\bar{\mathcal{K}}_{\infty}$, then $G'$ is finite.
\end{lem}

\begin{proof}
    By induction on the derived length of $G$, it is possible to assume that $G''$ is finite, so that, by replacing $G$ by $G/G''$, we may even suppose that $G''=\{1\}$. Thus Corollary \ref{meta} allows us to assume that the group $G$ is periodic. Put $A=G'$, which may be assumed infinite and let $g$ be any element of $G$ such that $[A,g]=[A,<g>]\neq \{1\}$. Consider the set $\pi$ consisting of all prime numbers $p$ for which $[A_p,g]\neq \{1\}$, and assume by contradiction that $[A,g]^p$ is infinite for all $p\in \pi$. Then $\frac{G}{[A,g]^p}$ has the property $\bar{\mathcal{K}}$, and hance $[A,g]/[A,g]^p$ is finite by Theorem A. It follows that $[A,g]/[A,g]^m$ is finite for each positive $\pi$-number $m$, and so $[A,g]$ has infinite exponent. Now form the chain $$A=A^{1!}\geq A^{2!}\geq \, ... \, \geq A^{n!} \geq \, .... $$ Since $A^{n!}$ is infinite for all $n$, there are finitely many subgroups of the form $[A^{n!},<g,A^{n!}>]=[A^{n!},g]$ and hence $$[A,g]^{k!}=[A,g]^{(k+1)!}= \, ...$$ for some $k>0$. This means that the subnormal subgroup $[A,g]^{k!}$ is divisible and hence trivial. This contradiction shows that $[A,g]^p$ is finite for some $p\in \pi$. Application of Lemma \ref{4.5} to the group $<A,g>/[A,g]^p$ shows that the subgroup $[A,g,g]$ is finite. Since $<A,g>/[A,g,g]$ is nilpotent, it follows from Corollary \ref{nilp} that $[A,g]$ is also finite. On the other hand, by the property $\bar{\mathcal{K}}_{\infty}$, there are only finitely many subgroups of the form $[A,<g,A>]=[A,g]$, and so $[A,G]$ is finite. As $\frac{G}{[A,G]}$ is nilpotent, we have by Corollary \ref{nilp} again that $A=G'$ is finite.
\end{proof}

\begin{lem}\label{rank}
    Let $G$ be a soluble group with the property $\bar{\mathcal{K}}_{\infty}$ and let $D$ be the largest divisible abelian periodic normal subgroup of $G$. If $D$ has infinite total rank, then $G'$ is finite.
\end{lem}

\begin{proof}
    Consider an element $g$ of $G$ and suppose first that $g$ has infinite order. If $x$ is any element of $D$, then the subgroup $H=<x,g>$ is finitely generated and metabelian, so it is residually finite; thus $H'$ is finite by Lemma \ref{rf}. On the other hand, there are finitely many subgroup of the type $[<x,g>,<g>]$, so $[D,g]$ is finite and divisible, and therefore $[D,g]=\{1\}$. Suppose now that $g$ has finite order. Since $D$ has infinite total rank, there is an infinite chain $$\{1\}<D_1< D_2< \, ...\,<D_n<\,...$$ of subgroups of $D$ such that each $D_n$ is divisible of finite rank and $D_n^g=D_n$. By hypothesis, there are finitely many subgroups of the form $[<g,D_n>,D_n]$ and so $[D_n,g]=[D_s,g]$ for any $s\geq n$ for some positive integer $n$. It follows that $D_s\leq D_nC_D(g)$ for any $s\geq n$, so that $C_D(g)$ must be infinite. Consequently, there are finitely many subgroups of the type $[<x>,<C_D(g),g>]=[x,g]^{<g>}$, where $x\in D$. Therefore $[D,g]$ is finite and hence trivial. We have shown that $D$ is contained in $Z(G)$, so that $G'$ is finite by Lemma \ref{centro} and Theorem A.
\end{proof}

We are now in a position to prove our main result about the property $\bar{\mathcal{K}}_{\infty}$.

\medskip

\noindent \textbf{Theorem B}

\noindent A soluble-by-finite group $G$ has the property $\bar{\mathcal{K}}_{\infty}$ if and only if it has finite commutator subgroup or it is a finite extension of a group of the type $p^{\infty}$ for some prime number $p$.

\begin{proof}
    The condition of the statement is obviously sufficient. In order to prove the converse statement, first suppose that the group is a soluble $\bar{\mathcal{K}}_{\infty}$-group and that $G'$ is infinite. Let $D$ be the largest divisible abelian periodic normal subgroup of $G$. Then $D$ is a non-trivial group with finite total rank by Lemmas \ref{4.7} and \ref{rank}. For each positive integer $n$, let $D[n]$ be the subgroup consisting of all elements $a$ of $D$ such that $a^n=1$. Assume that $G$ contains an element $x$ of infinite order. Then the subgroup $[<x>,<x,D[n]>]=[D[n],x]$ is finite for every $n$ and there are only finitely many of these subgroups, so that $[D,x]$ is finite and hence $[D,x]=\{1\}$. As in the proof of Lemma \ref{per} it can be shown that the derived subgroup of a soluble-by-finite $\bar{\mathcal{K}}_{\infty}$-group is periodic, so that $G$ is generated by elements of infinite order; thus $D\leq Z(G)$ and $Z(G)$ is infinite, a contradiction by Corollary \ref{cenk}. Therefore, $G$ is periodic. Consider the centralizer $C=C_G(D)$. Then $D$ is contained in $Z(C)$, and so $C'$ is finite by Corollary \ref{cenk}. Let $K/C'$ be the subgroup generated by all elements of prime order in $C/C'$. If $K$ is infinite, then the group $G/K$ has the property $\bar{\mathcal{K}}$, and so $(G/K)'$ is finite by Theorem A; thus $[D,G]\leq K$ and hence $[D,G]={1}$, which is a contradiction because $Z(G)$ must be finite. Therefore $K$ is finite and it follows that $C$ is a Černikov group. Moreover, $G/C$ is isomorphic to a periodic group of automorphisms of $D$, so it is finite (see [\cite{robinson}, Part 1, p. 85]) and $G$ is a Černikov group in this case. Suppose now that $G$ is soluble-by-finite. Let $S$ be a soluble normal subgroup of finite index in $G$, and assume by contradiction that $G$ is not a Černikov group. Then $S$ cannot be a Černikov group and so for each element $g$ of $G$, the subgroup $[<S,g>,S]=[S,g]S'$ is finite. Since there are only finitely many subgroups of the form $[<S,g>,S]$, it follows that $[S,G]$ is finite. On the other hand, the group $G/[S,G]$ is central-by-finite, so $G'/[S,G]$ is finite and $G'$ itself is finite. This contradiction shows that $G$ is a Černikov group.
    
    Let $J$ be the largest divisible abelian periodic normal subgroup of $G$. Assume for a contradiction that $J$ is not of the type $p^{\infty}$ for any prime number $p$ and that $G'$ is infinite. Thus there exists in $J$ a strictly ascending chain of infinite subgroups $$J_1<J_2<\,...\, <J_n<J_{n+1}<\,....$$ Let $x$ be an element of $G$ such that $C_{J}(x)<J$. It follows from the property $\bar{\mathcal{K}}_{\infty}$ that there exists a positive integer $m$ such that $[<x>,J_m]=[<x>,J_n]$ for every positive integer $n\geq m$. In particular, fixed $n\geq m$, since $J$ is an abelian normal subgroup of $G$, we have that $[x,J_n^{<x>}]=[x,J_m^{<x>}]=\{[x,a] \ |\ a\in J_m^{<x>}\}$. Then $J_n\leq <J_m,x>C_{<x>J}(x)$ for each $n\geq m$ and hence the subgroup $C_{<x>J}(x)$ is infinite. Consider now $C_J(x)$. Then $|C_{<x>J}(x):C_J(x)|$ is finite, and so $C_J(x)$ is infinite. Let $H$ be the subgroup $<x,C_G(J)>$. Thus $C_J(x)\leq Z(H)$, so $H'$ is finite by Corollary \ref{cenk}, and hence $[x,J]=\{1\}$. Then $J\leq Z(G)$ and, again by Corollary \ref{cenk}, $G'$ is finite. This contradiction completes the proof of the Theorem. 
\end{proof}


\end{document}